\documentclass[a4paper]{amsart}

\usepackage[all]{xy}
\usepackage{amsmath}
\usepackage{amssymb}
\usepackage{amsthm}
\usepackage{latexsym}
\usepackage{enumerate}
\usepackage{prettyref}
\usepackage{hyperref}

\newcommand{\Ao}{\ensuremath{\mathbb{A}^1}}

\newcommand{\topos}[1]{\ensuremath{\mathcal{S}hv(#1)}}
\newcommand{\simplicial}[1]{\ensuremath{\Delta^{op}#1}}
\newcommand{\sms}{\ensuremath{\operatorname{Sm}_S}}
\newcommand{\smk}{\ensuremath{\operatorname{Sm}_k}}

\newcommand{\Hom}{\ensuremath{\operatorname{Hom}}}
\newcommand{\set}{\ensuremath{\mathcal{S}et}}

\newcommand{\id}{\ensuremath{\operatorname{id}}}

\newcommand{\diagram}[1]{\ensuremath{\mathcal{#1}}}
\newcommand{\category}[1]{\ensuremath{\mathcal{#1}}}
\newcommand{\colim}{\ensuremath{\operatornamewithlimits{colim}}}
\newcommand{\hocolim}{\ensuremath{\operatornamewithlimits{hocolim}}}
\newcommand{\inthom}{\ensuremath{\mathbf{Hom}}}

\newcommand{\classify}[1]{\ensuremath{B\haut({#1})}}

\newcommand{\haut}{\ensuremath{\operatorname{hAut}_\bullet}}

\newtheorem{mainthm}{Theorem}
\newtheorem{theorem}{Theorem}[section]
\newrefformat{thm}{\hyperref[{#1}]{Theorem~\ref*{#1}}}
\newtheorem{definition}[theorem]{Definition}
\newrefformat{def}{\hyperref[{#1}]{Definition~\ref*{#1}}}
\newtheorem{lemma}[theorem]{Lemma}
\newrefformat{lem}{\hyperref[{#1}]{Lemma~\ref*{#1}}}
\newtheorem{proposition}[theorem]{Proposition}
\newrefformat{prop}{\hyperref[{#1}]{Proposition~\ref*{#1}}}
\newtheorem{corollary}[theorem]{Corollary}
\newrefformat{cor}{\hyperref[{#1}]{Corollary~\ref*{#1}}}
\newtheorem{remark}[theorem]{Remark}
\newrefformat{rem}{\hyperref[{#1}]{Remark~\ref*{#1}}}
{
\newtheorem{examplecore}[theorem]{Example}}
\newrefformat{ex}{\hyperref[{#1}]{Example~\ref*{#1}}}

\begin{document}

\title{Fibre Sequences and Localization of Simplicial Sheaves} 
\date{Revised version, April 2012}

\author{Matthias Wendt}

\address{Matthias Wendt, Mathematisches Institut,
Uni\-ver\-si\-t\"at Freiburg, Eckerstra\ss{}e 1, 79104,
  Freiburg im Breisgau, Germany}
\email{matthias.wendt@math.uni-freiburg.de}


\begin{abstract}
In this paper, we discuss the theory of quasi-fibrations in proper Bousfield 
localizations of model categories of simplicial sheaves.
We provide a construction of fibrewise localization and use this
construction to generalize a criterion for locality of fibre sequences
due to Berrick and Dror Farjoun. The results allow a
better understanding of unstable $\Ao$-homotopy theory.
\end{abstract}

\maketitle
\setcounter{tocdepth}{1}
\tableofcontents

\section{Introduction}

In this paper, we discuss aspects of Bousfield localization for simplicial 
sheaves. One of the main phenomena of interest is the behaviour of fibrations
resp. fibre sequences under a Bousfield localization. In general, fibrations
and fibre sequences are not preserved by a Bousfield localization, and it
is an interesting question to find suitable criteria under which they are
preserved. An extensive discussion of issues related to this question
can be found in \cite{farjoun:1996:cellular}. A general criterion for
locality of  fibre sequences in nullifications has been obtained by
Berrick and  Dror Farjoun in \cite{berrick:farjoun:2003:null}. The
main goal of this paper  is to provide a generalization of this result
to the setting of simplicial sheaves. It should be pointed out that
the methods heavily use homotopy  pullbacks and therefore only apply
to the case where the Bousfield localization is right proper.

The main tool used in the present work is an analogue of the theory of 
the quasi-fibrations of Dold and Thom \cite{dold:thom}. On the one hand, 
quasi-fibrations behave like fibrations in that point-set and homotopy fibres
agree - in particular, quasi-fibrations give rise to fibre sequences and 
hence long exact homotopy sequences. On the other hand, quasi-fibrations are
much more flexible than fibrations. In the setting of categories of simplicial
sheaves, the sharp maps of Rezk \cite{rezk:1998:sharp} provide a replacement
for quasi-fibrations for model categories of simplicial sheaves. This theory
has been used in \cite{classify} to produce classifying spaces for fibre
sequences of simplicial sheaves. In the present paper, we consider the
notion of universally $f$-local  maps, cf. \prettyref{def:univfloc},
 in (proper) Bousfield localizations of model categories of simplicial 
sheaves. This notion as well as the basic assertions in
\prettyref{sec:fibseq} are due to Jardine and were suggested to me as a
correction to a mistake in an earlier version of this paper. 
Jardine's definition of universally $f$-local maps is equivalent to
the definition of sharp maps given by Rezk in \cite{rezk:1998:sharp},
and the assertions in  \prettyref{sec:fibseq}  show that universally
$f$-local maps provide a good theory of ``$f$-local quasi-fibrations''.

There are two simple reasons why the calculus of universally $f$-local maps
works in Bousfield localizations of simplicial sheaves: on the one hand,
one can use the homotopy colimit decomposition and homotopy distributivity 
of simplicial sheaves. On the other hand, the properness of the localized
model structure has the important consequence that a simplicial 
quasi-fibration over an $f$-local base is an $f$-local quasi-fibration. 

With the $f$-local quasi-fibrations, it is possible to give a construction 
of fibrewise $f$-localization. The construction we give in 
\prettyref{sec:fibwise} is almost a direct translation of the fibrewise 
localization in the category of simplicial sets - again the main technical
tools are the homotopy colimit decomposition and the properness of the local
model structure. 

Once we have a working construction of fibrewise localization, we can 
almost directly translate the criterion of Berrick and Dror Farjoun
to the simplicial sheaf setting. The result is then the following, 
cf. \prettyref{thm:bfthm}: 

\begin{mainthm}
Let $T$ be a site and let $f:X\rightarrow Y$ be a morphism of
simplicial sheaves in $\simplicial{\topos{T}}$. Assume that the
$f$-local model structure is proper. Let $p:E\rightarrow B$ be a
morphism of simplicial sheaves. 

We denote by $\overline{p}:\overline{E}\rightarrow B$ the fibrewise
$f$-localization of $p$, and by $j:B\rightarrow L_fB$ an $f$-local fibrant 
replacement of $B$.
The following are equivalent, where (iv) only makes sense if
$p:E\rightarrow B$ is locally trivial:
\begin{enumerate}[(i)]
\item The map $p:E\rightarrow B$ is universally $f$-local. 
\item The fibrewise localization $\overline{p}:\overline{E}\rightarrow
  B$ is universally $f$-local.
\item For each simplex $\sigma:\Delta^n\times U\rightarrow L_fB$, the
  following canonical diagram is a simplicial homotopy pullback:
\begin{center}
  \begin{minipage}[c]{10cm}
    \xymatrix{
      (\Delta^n\times U)\times_{L_fB}\overline{E} \ar[r] \ar[d]
      & \overline{p}^{-1}(\sigma) \ar[d] \\
      (\Delta^n\times U)\times_{L_fB}B \ar[r] &
      \Delta^n\times U.
    }
  \end{minipage}
\end{center}
Here $\overline{p}^{-1}(\sigma)$ denotes the fibre of the fibrewise
localization over $\sigma$, cf. \prettyref{def:fibloc1}.
\item For each simplex $\sigma:\Delta^n\times U\rightarrow L_fB$, the
  composition
$$(\Delta^n\times U)\times_{L_fB}B=j^{-1}(\sigma)\rightarrow
B\rightarrow \classify{F} \rightarrow \classify{L_fF}
$$
factors (in the simplicial homotopy category) through the projection
$(\Delta^n\times U)\times_{L_fB}B\rightarrow \Delta^n\times U$.
\end{enumerate}
\end{mainthm}

It should be noted that the above result specializes exactly to 
\cite[Theorem 4.1]{berrick:farjoun:2003:null}. The additional complication 
in the formulation of the above theorem is due to the fact that the homotopy 
colimit decomposition of a simplicial sheaf allows to decompose a simplicial
sheaf $X$ as the homotopy colimit of its simplices 
$\Delta^n\times U\rightarrow X$, but the spaces $\Delta^n\times U$ are not
necessarily contractible. However, the interpretation of the above theorem
is still the same: a map of simplicial sheaves $p:E\rightarrow B$ is 
universally $f$-local if the restriction of its fibrewise localization 
$\overline{p}:\overline{E}\rightarrow B$ to non-local parts of $B$ is 
``trivial''. Here non-local parts of $B$ are fibres of $j:B\rightarrow L_fB$ 
over simplices $\Delta^n\times B\rightarrow L_fB$, and ``trivial'' means
that the corresponding map is a pullback of a map over $\Delta^n\times U$. 

As an interesting application, we arrive at conditions when morphisms
induce fibre sequences in $\Ao$-homotopy theory. In the case where the
morphisms are locally trivial in the Nisnevich topology, the homotopy
theory criteria reduce to a simple condition on the sheaf of homotopy
self-equivalences of the fibre.

\begin{mainthm}
Let $F$ be a simplicial sheaf on $\smk$.
If $\pi_0\haut L_{\Ao}F$ is a strongly $\Ao$-invariant sheaf of groups, then
any morphism $p:E\rightarrow B$ which is locally trivial in the
Nisnevich topology with fibre $F$ is universally $f$-local. In
particular, there are  $\Ao$-local fibre sequences $F\rightarrow
E\rightarrow B$ for any choice of base point of $B$.
\end{mainthm}

\emph{Structure of the paper:} In \prettyref{sec:classify}, we recall 
preliminaries on model structures on categories of simplicial sheaves, in 
particular homotopy distributivity and homotopy colimit decomposition. 
In \prettyref{sec:prelims}, we recall preliminaries on the Bousfield
localization of simplicial sheaves, in particular regarding propernesss of
the localized model structure. Then \prettyref{sec:fibseq} provides an 
exposition of Jardine's universally $f$-local maps and their properties. 
These properties are used in \prettyref{sec:fibwise} to construct
a fibrewise localization for fibrations of simplicial sheaves. 
\prettyref{sec:bfthm} provides the main characterization result for 
universally $f$-local maps which generalizes the result of Berrick and Dror 
Farjoun. Finally, \prettyref{sec:appl} discusses applications to
$\Ao$-homotopy theory.

\emph{Acknowledgements:} 
The results presented here are taken from my PhD thesis \cite{thesis}
which was  supervised by Annette Huber-Klawitter. I would like to use
the opportunity to thank her for her encouragement and interest in my 
work. I would also like to thank Rick Jardine for pointing out a
mistake in an earlier version, and for his extremely helpful letter on
universally $f$-local maps. All the material in \prettyref{sec:fibseq}
is due to Jardine and is included in this paper with his permission. The
present paper would not have its present form without his input.

\section{Preliminaries on simplicial sheaves}
\label{sec:classify}

\subsection{Model structures on simplicial sheaves}

We will be working in categories of simplicial sheaves. The underlying
site is usually denoted by $T$, the category of sheaves on it by
$\topos{T}$, and the category of simplicial sheaves by
$\simplicial{\topos{T}}$. On this category, there are several model
structures all yielding the same homotopy theory. We will use the
injective model structure, cf. 
\cite[Theorems 18 and
27]{jardine:1996:boolean}.

\begin{theorem}
Let $\mathcal{E}$ be a topos. Then the category
$\simplicial{\mathcal{E}}$ of simplicial objects in $\mathcal{E}$ has
a model structure, where the
\begin{enumerate}[(i)]
\item cofibrations are monomorphisms,
\item weak equivalences are detected on a fixed Boolean localization,
\item fibrations are determined by the right lifting property. 
\end{enumerate}

Moreover, the above definition of weak equivalences does not depend on
the Boolean localization. 
\end{theorem}

The following proposition recalls the basic properties of this model
structure. Existence is proved in  \cite[Theorems 18 and
27]{jardine:1996:boolean}. Properness and simpliciality
are proven in \cite[Theorem 24]{jardine:1996:boolean}. 
Cellularity is proven in \cite[Theorem 1.4]{hornbostel:2006}.

\begin{proposition}
\label{prop:injcomb}
Let $T$ be any Grothendieck site. Then the injective model
structure of Jardine on the category of (pre-)sheaves of simplicial
sets on $T$ is a proper simplicial and cellular model
structure. 
\end{proposition}

\subsection{Homotopy pullbacks}

Recall that a commutative square
\begin{center}
  \begin{minipage}[c]{10cm}
    \xymatrix{
      A\ar[d]\ar[r] & C\ar[d] \\
      B\ar[r] & D
    }
  \end{minipage}
\end{center}
in a model category $\mathcal{C}$ is called a homotopy pullback if for
some factorization $C\rightarrow \tilde{C}\rightarrow D$ as a  trivial
cofibration $C\rightarrow\tilde{C}$ and a  fibration
$\tilde{C}\rightarrow D$, the induced map
$A\rightarrow B\times_D\tilde{C}$ is a weak equivalence. 
This notion is only well-defined if the model category $\mathcal{C}$
is proper, cf. \cite[Section
II.8]{goerss:jardine:1999:simplicial}. As homotopy pullbacks play a
major role in this paper, all the model categories in sight will be
assumed to be proper. 

An important special case of homotopy pullbacks are those of the form 
\begin{center}
  \begin{minipage}[c]{10cm}
    \xymatrix{
      F\ar[d]\ar[r] & E\ar[d] \\
      \ast\ar[r] & B,
    }
  \end{minipage}
\end{center}
i.e. in which one of the factors is contractible. Such pullbacks are
basically the same thing as fibre sequences. As there is always a
problem with base points and different homotopy types of simplices in
categories of sheaves, it is better to talk generally about homotopy
pullbacks rather than fibre sequences.

\subsection{Homotopy distributivity and colimit decomposition}

Next, we repeat several basic statements on the behaviour of homotopy
limits and colimits in categories of simplicial sheaves.  The main
result needed is the homotopy distributivity of Rezk,
cf. \cite{rezk:1998:sharp}. Results and
preliminaries can also be found in \cite[Section 2]{classify}.

Recall that a diagram $\diagram{X}:\mathcal{I}\rightarrow
\mathcal{C}$ in a model category $\mathcal{C}$ is called
\emph{homotopy colimit diagram} if 
the canonical map $\hocolim\diagram{X}\rightarrow \colim\diagram{X}$
is a weak equivalence. We can now recall the definition of homotopy
distributivity: let $\category{C}$ be a simplicial
model category, let $\category{I}$ be a small category, and let
$f:\diagram{X}\rightarrow \diagram{Y}$ be a morphism of
$\category{I}$-diagrams in $\category{C}$. The diagrams we are most
interested in are the following: 

For any $i\in\category{I}$, we have a commutative square
\begin{center}
  \begin{minipage}[c]{10cm}
    \begin{equation}
      \label{eq:distrib1}
      \xymatrix{
        \diagram{X}(i) \ar[r] \ar[d]_{f(i)} & \colim_\category{I}
        \diagram{X} \ar[d] \\ 
        \diagram{Y}(i) \ar[r] & \colim_\category{I} \diagram{Y}.
      }
    \end{equation}
  \end{minipage}
\end{center}

Moreover, for any $\alpha:i\rightarrow j$ in $\category{I}$ we have a 
commutative square
\begin{center}
  \begin{minipage}[c]{10cm}
    \begin{equation}
      \label{eq:distrib2}
      \xymatrix{
        \diagram{X}(i) \ar[r]^{\diagram{X}(\alpha)} \ar[d]_{f(i)} &
        \diagram{X}(j)
        \ar[d]^{f(j)} \\ 
        \diagram{Y}(i) \ar[r]_{\diagram{Y}(\alpha)} &
        \diagram{Y}(j). 
      }
    \end{equation}
  \end{minipage}
\end{center}

\begin{definition}[Homotopy Distributivity]
In the above situation, we  say that $\category{C}$ 
satisfies \emph{homotopy distributivity} if for any morphism
$f:\diagram{X}\rightarrow \diagram{Y}$ of 
$\category{I}$-diagrams in  $\category{C}$ for which $\diagram{Y}$ is
a homotopy 
colimit diagram, i.e.  $\hocolim_\category{I} \diagram{Y}\rightarrow
\colim_\category{I} \diagram{Y}$ is a weak equivalence, the following
two properties hold: 
\begin{enumerate}[(HD i)]
\item If each square of the form (\ref{eq:distrib1}) is a homotopy
  pullback, then $\diagram{X}$ is a homotopy colimit diagram.
\item If $\diagram{X}$ is a homotopy colimit diagram, and each
  diagram of the form (\ref{eq:distrib2}) is a homotopy  pullback, then
  each diagram of the form (\ref{eq:distrib1}) is also a homotopy pullback.  
\end{enumerate}
\end{definition}

\begin{proposition}
Let $T$ be a site. Then homotopy distributivity holds in the model
category $\simplicial{\topos{T}}$. 
\end{proposition}

The main consequence of homotopy distributivity is the canonical
homotopy colimit decomposition of morphisms of simplicial
sheaves. This allows to write the source of a morphism as homotopy
colimit of its fibres over simplices of the target. 
We first recall the homotopy colimit decomposition for simplicial
sets: for a simplicial set $X$, we can consider its category of
simplices $\mathbf{\Delta}\downarrow X$ whose objects are morphisms 
$\Delta^n\rightarrow X$ and whose morphisms are the obvious
commutative triangles. The notation $\mathbf{\Delta}\downarrow X$ is
chosen because the category of simplices is the comma category of
objects under the standard simplices. For a morphism of simplicial
sets $f:X\rightarrow Y$, one can then associate a functor
$f^{-1}:\mathbf{\Delta}\downarrow Y\rightarrow\simplicial{\set}$ by
mapping a simplex $\sigma:\Delta^n\rightarrow Y$ to the simplicial set
$f^{-1}(\sigma)$ defined by the following pullback diagram: 
\begin{center}
  \begin{minipage}[c]{10cm}
    \xymatrix{
      f^{-1}(\sigma)\ar[r]\ar[d] & X\ar[d]^f\\
      \Delta^n\ar[r]_\sigma&Y.
    }
  \end{minipage}
\end{center}
There is a canonical morphism of simplicial sets $\hocolim
f^{-1}\rightarrow X$ which is a weak equivalence, cf. \cite[Lemma
IV.5.2]{goerss:jardine:1998:localization}. A similar statement holds
for simplicial sheaves. The right notion to formulate it is the
canonical homotopy colimit decomposition for objects in a
combinatorial model category, cf. \cite{dugger:2001:combinatorial}.  
For the convenience of the reader we recall notation and (a
generalization of) a lemma already formulated in \cite[Section
2.8]{classify}.  
Let $\category{C}$ be a combinatorial model category, $\category{T}$
be a small category. For any functor
$I:\category{T}\rightarrow\category{C}$ and a fixed cosimplicial
resolution
$\Gamma_I:\category{T}\rightarrow\mathbf{\Delta}\category{C}$, we 
obtain a functor
$\category{T}\times\mathbf{\Delta}\rightarrow\category{C}:(U,[n])\mapsto 
\Gamma(n)(U)$ which replaces the standard cosimplicial object
$\mathbf{\Delta}$ in $\simplicial{\set}$ above.  For any object $X$,
we can consider the over-category
$(\category{T}\times\mathbf{\Delta}\downarrow X)$ and the canonical
diagram $(\category{T}\times\mathbf{\Delta}\downarrow 
X)\rightarrow \category{C}:\Gamma(n)(U)\mapsto U\times\Delta^n$ which
is the proper replacement for the category of simplices.

The following lemma was formulated in \cite{classify} only for a
fibration of fibrant simplicial sheaves. In this special case, its
proof is an application of homotopy distributivity.

\begin{lemma}
\label{lem:decompfib}
Let $T$ be a site, and let 
$p:E\rightarrow B$ be a morphism of simplicial sheaves. Then
$p$ is weakly   equivalent to the morphism of simplicial sheaves
\begin{displaymath}
\hocolim\diagram{F}\rightarrow \hocolim(T\times
\mathbf{\Delta}\downarrow B), 
\end{displaymath}
where $(T\times\mathbf{\Delta}\downarrow B)$ is the canonical diagram
associated to some fixed cosimplicial resolution, and the diagram
$\diagram{F}$ is the diagram of fibres: 
$$
\diagram{F}:(T\times\mathbf{\Delta}\downarrow B)\rightarrow\simplicial{\topos{T}}:
(U\times\Delta^n\rightarrow B)\mapsto(U\times\Delta^n)\times_B E.
$$
\end{lemma}

\begin{proof}
We have a composition of morphisms 
$$\hocolim \diagram{F}\rightarrow\colim\diagram{F} \rightarrow E.
$$
The second morphism is an isomorphism by the distributivity in categories of
sheaves, cf. e.g. \cite[Proposition 3.7]{rezk:1998:sharp}. It then suffices 
to prove that the diagram $\diagram{F}$ is a homotopy colimit
diagram. 

If the topos has enough points, this can be checked on points, cf. 
\cite[Corollary 2.10]{classify}. For a point $x$ of the topos $\topos{T}$, 
the diagram $x^\ast(\mathcal{F})$ is the diagram of fibres of the simplicial
set map $x^\ast(p):x^\ast(E)\rightarrow x^\ast(B)$:
$$
x^\ast(\diagram{F}):(\mathbf{\Delta}\downarrow x^\ast(B))\rightarrow
\simplicial{\set}:
(\sigma:\Delta^n\rightarrow B)\mapsto (x^\ast(p))^{-1}(\sigma)=\Delta^n\times_B E.
$$
In particular, the 
composition $x^\ast(\hocolim\diagram{F})\rightarrow x^\ast(\colim\diagram{F})
\rightarrow E$ is the composition $\hocolim (x^\ast(\mathcal{F}))
\rightarrow \colim(x^\ast(\mathcal{F}))\rightarrow x^\ast(E)$. But the latter 
is known to be a homotopy colimit diagram, cf. 
\cite[IV.5.2]{goerss:jardine:1999:simplicial}.

The same argument as above shows that the assertion is true in the
presheaf category, because colimits (and therefore homotopy colimits)
of simplicial presheaves are computed pointwise. The general case then
follows from the properties of the sheafification functor.
\end{proof}

\begin{corollary}
\label{cor:2jardine}
Consider the following commutative triangle, in which $p_1$ and $p_2$
are fibrations:
\begin{center}
  \begin{minipage}[c]{10cm}
    \xymatrix{
      E_1\ar[rr]^f\ar[dr]_{p_1} && E_2\ar[dl]^{p_2}\\
      &B
    }
  \end{minipage}
\end{center}
Then the morphism $f$ is a weak equivalence if one of the following
holds:
\begin{enumerate}[(i)]
\item The induced morphisms $p_1^{-1}(\sigma)\rightarrow
  p_2^{-1}(\sigma)$ are weak equivalences for all objects
  $\sigma:\Delta^n\times U\rightarrow B$ of the category of simplices
  $(\mathbf{\Delta}\times T)\downarrow B$. 
\item The induced morphisms $p_1^{-1}(x)\rightarrow p_2^{-1}(x)$ are
  weak equivalences for all $x:U\rightarrow B$.
\end{enumerate}
\end{corollary}

\begin{proof}
(i) We have a commutative square 
\begin{center}
  \begin{minipage}[c]{10cm}
    \xymatrix{
      \hocolim\mathcal{F}_1\ar[r]\ar[d]&\hocolim\mathcal{F}_2\ar[d] \\
      E_1\ar[r]_f&E_2
    }
  \end{minipage}
\end{center}
in which the two vertical morphisms are weak equivalences by
\prettyref{lem:decompfib}. The top horizontal morphism is a homotopy
colimit of weak equivalences, therefore the bottom horizontal morphism
is a weak equivalence. 

(ii) follows from (i) by considering the following diagram in which
all squares are homotopy pullbacks:
\begin{center}
  \begin{minipage}[c]{10cm}
    \xymatrix{
      p_i^{-1}(\sigma(v))\ar[r]^\simeq \ar[d] &
      p_i^{-1}(\sigma)\ar[r]\ar[d] & E_i\ar[d]^{p_i}\\
      U\ar[r]_v^\simeq &\Delta^n\times U\ar[r]_\sigma &B.
    }
  \end{minipage}
\end{center}
The right square is a homotopy pullback because $p_i$ is a fibration
(and the model structure is proper), the left because there are two
parallel weak equivalences. It then suffices to check weak
equivalences after restriction to vertices of simplices.
\end{proof}

\section{Preliminaries on localization functors}
\label{sec:prelims}

\subsection{Bousfield Localization} We repeat the standard definitions of
local objects and local weak equivalences. These definitions
can be found in \cite{farjoun:1996:cellular,hirschhorn:2003:modelcats} for the
case of simplicial sets, and in \cite{morel:voevodsky:1999:a1} for the case of
simplicial sheaves. 

Let $\category{C}$ be a model category, and let $f:X\rightarrow Y$ be a
morphism of cofibrant objects. 

\begin{definition}[Local Objects, Weak Equivalences]
An object $A\in \category{C}$ is called \emph{$f$-local} if $A$ is fibrant
and the following morphism is a bijection for each $B\in
\operatorname{Ho}\mathcal{C}$:  
\begin{displaymath}
\hom_{\operatorname{Ho}\mathcal{C}}(B\times Y,A)\rightarrow
\hom_{\operatorname{Ho}\mathcal{C}}(B\times X,A).  
\end{displaymath}

A morphism $g:A\rightarrow B\in\category{C}$ is called an \emph{$f$-local
  weak equivalence} if for any $f$-local object $C$, the following morphism is
a bijection:
\begin{displaymath}
\hom_{\operatorname{Ho}\mathcal{C}}(B,C)\rightarrow
\hom_{\operatorname{Ho}\mathcal{C}}(A,C). 
\end{displaymath}
\end{definition}

\begin{remark}
\begin{enumerate}[(i)]
\item
The above is the definition of local given in
\cite{morel:voevodsky:1999:a1}. It is easy to check that it coincides
with the definition in \cite{goerss:jardine:1998:localization}, where
one requires a weak equivalence of simplicial sets:
\begin{displaymath}
\Hom(B,C)\rightarrow
\Hom(A,C). 
\end{displaymath}
This in turn is equivalent to requiring weak equivalences on internal
homs:
\begin{displaymath}
\inthom(B,C)\rightarrow
\inthom(A,C). 
\end{displaymath}
\item Note that there is a difference between pointed and
  unpointed. The definition above is for a general model category,
  using unpointed mapping spaces. In a pointed model category, one
  uses the pointed mapping spaces. For connected objects both notions
  coincide.
\end{enumerate}
\end{remark}

Of course, one can consider more general localizations, i.e. localizations
with respect to a \emph{set} of maps as in \cite[Section
2.2]{morel:voevodsky:1999:a1}, or homology localization as in
\cite[Section 3]{goerss:jardine:1998:localization}.  
If $f$ is null-homotopic such a localization is
also called \emph{nullification}, and we also use $L_W$ to denote the
corresponding localization functor. The 
most important applications we have in mind are the $\Ao$-nullification
functors $L_{\Ao}$ on $\simplicial{\topos{\sms}}$.

\subsection{Localization Functors} This paragraph repeats the theorem
on existence and universality of localization functors for simplicial
sheaves. Most of the elementary facts in
\cite[1.A.8]{farjoun:1996:cellular}  are easy
consequences of this theorem, which is proved in \cite[Theorem
2.2.5]{morel:voevodsky:1999:a1} and in similar form in \cite[Theorem
4.4]{goerss:jardine:1998:localization}.  

We start recalling necessary  definitions related to localization
functors in a general model category. 

\begin{definition}
A functor $F:\category{C}\rightarrow\category{C}$ is called
\emph{coaugmented} if there is a natural transformation
$j:\id_\category{C}\rightarrow F$. A coaugmented functor $F$ is called 
\emph{idempotent} if the two natural maps
$j_{FA},Fj_A:FA\rightrightarrows FFA$ are weak equivalences and
homotopic to each other. The coaugmentation map $j_A$ is homotopy
universal with respect to maps into local spaces if any map
$A\rightarrow B$ into a local space $T$ factors uniquely (up to
homotopy) through $j_A:A\rightarrow FA$. The functor $F$ is called
\emph{simplicial} if it is 
compatible with the simplicial structure, i.e. if there exist functorial
morphisms $\sigma:(FA)\otimes K\rightarrow F(A \otimes K)$ for any object
$A\in\category{C}$ and any simplicial set $K$. These morphisms have to satisfy
some rather obvious conditions described in \cite[Definition
1.C.8]{farjoun:1996:cellular}. The functor $F$ is called \emph{continuous} if
it induces a morphism on inner function spaces  
\begin{displaymath}
\inthom(A,B)\rightarrow \inthom(FA,FB),
\end{displaymath}
which is compatible with composition. 
\end{definition}

We recall the existence of
localizations for simplicial sheaf categories from
\cite[Theorem 4.4]{goerss:jardine:1998:localization}, which is
the proper generalization of \cite[Theorem A.3]{farjoun:1996:cellular}.    
The existence of the $f$-local model structure is proven in \cite[Theorem
4.8]{goerss:jardine:1998:localization}. Note that the existence of
localizations for simplicial sheaves is a global result, in the sense that
it does not simply follow from the existence of localizations of simplicial
sets by looking at the points of the topos.

\begin{theorem}
\label{thm:locmodel}
Let $f:X\rightarrow Y$ be any cofibration in \simplicial{\topos{T}}
and suppose $\alpha$ is an infinite cardinal which is an upper bound
for the cardinalities of both $Y$ and the set of morphisms of
$T$. Then there exists a functor $L_f$, called the
\emph{$f$-localization functor}, which is coaugmented and
homotopically idempotent. Any two such functors are naturally weakly
equivalent to each other. The map $A\rightarrow L_fA$ is a
homotopically universal map to $f$-local spaces. Moreover, $L_f$ can
be chosen to be simplicial and continuous. 

There is a simplicial model structure on $\simplicial{\topos{T}}$
where the cofibrations are monomorphisms, weak equivalences are $f$-local weak
equivalences and fibrations are defined via the right lifting property.
\end{theorem}

\subsection{Properness}

In \cite[Chapter 3]{hirschhorn:2003:modelcats}, Bousfield localizations
of general model categories are investigated. As shown in
\cite[Proposition 3.4.4 and Theorem 4.1.1]{hirschhorn:2003:modelcats},
left Bousfield localizations preserve left properness, i.e. the left
Bousfield localization of a left proper model category is again left 
proper. 

The $f$-local model structure for a morphism $f:X\rightarrow Y$ is not
in general right proper. It is known \cite[Theorem
  A.5]{jardine:2000:motivic}, that the $f$-local model structure is
proper if $f$ is of the form $\ast\rightarrow I$, i.e. $L_f$ is a
nullification. 
A special case of this is the properness of the homotopy theory of a
site with interval, which is proved in 
\cite[Theorem 2.2.7 and Section 2.3]{morel:voevodsky:1999:a1}. 

We mention again that we will be working a lot with $f$-local homotopy
pullbacks, i.e. homotopy pullbacks in the $f$-local model structure. 
Therefore, throughout the rest of the paper,  we will assume that
$f:X\rightarrow Y$ is a morphism of simplicial sheaves on a site $T$
such that the $f$-local model structure on $\simplicial{\topos{T}}$ is
proper. Most of the time, this will be explicitly mentioned anyway.

\section{$f$-local sharp maps: universally $f$-local maps}
\label{sec:fibseq}

In this section, we will discuss a class of maps called
\emph{universally $f$-local maps}, which should be thought of
``$f$-local quasi-fibrations'' - they are not necessarily fibrations
in the $f$-local model structure but give rise to $f$-local fibre
sequences. 

\begin{definition}
\label{def:univfloc}
Let $T$ be a site and let $f:X\rightarrow Y$ be a morphism of
simplicial sheaves in $\simplicial{\topos{T}}$ such that the $f$-local
model structure is proper. A morphism
$p:E\rightarrow B$ of simplicial sheaves is called \emph{universally
  $f$-local} if for any representable $U\in T$ and any simplex
$\sigma:\Delta^n\times U\rightarrow B$ the following pullback diagram
is an $f$-local homotopy pullback:
\begin{center}
  \begin{minipage}[c]{10cm}
    \xymatrix{
      p^{-1}(\sigma)\ar[r]\ar[d] & E\ar[d]^p\\
      \Delta^n\times U\ar[r]_\sigma&B.
    }
  \end{minipage}
\end{center}
\end{definition}

\begin{remark}
\begin{enumerate}[(i)]
\item As $\sigma:\Delta^n\times U\rightarrow B$ ranges over the
  various simplices of the base simplicial sheaf $B$, $p^{-1}(\sigma)$
  ranges through the possible fibres of $p:E\rightarrow B$. The
  definition of universally $f$-local map makes sure that all these
  objects - which could be called \emph{local homotopy fibres of $p$}
  - have the right homotopy type. Note however, that the different
  representable objects $U\in T$ usually have different homotopy
  types in $\simplicial{\topos{T}}$, so that for different $U_1,U_2$,
  the fibres over $\Delta^n\times U_1\rightarrow B$ and
  $\Delta^m\times U_2\rightarrow B$ will usually not be weakly
  equivalent - this is the major   difference to the case of
  simplicial sets where all simplices $\Delta^n$ are weakly equivalent
  to the point.
\item
Note that if $\mathcal{C}$ is a model category with a terminal
object $\operatorname{pt}$, $f:X\rightarrow Y$ is a morphism in
$\mathcal{C}$,  $p:E\rightarrow B$ is universally $f$-local and
$x:\operatorname{pt}\rightarrow E$ is a choice of base-point of $E$,
then $(p^{-1}(x),x)\rightarrow (E,x)\rightarrow (B,p(x))$ is a fibre
sequence in the sense of \cite[Definition
6.2.6]{hovey:1998:modelcats} in the pointed model category
$(\mathcal{C},\operatorname{pt})$. In particular, a universally
$f$-local map induces long exact homotopy sequences for any choice of
base points. 
\end{enumerate}
\end{remark}

Even more is true: any pullback involving a universally $f$-local map
is an $f$-local homotopy pullback, provided of course the respective
$f$-local model structure is proper. All arguments in the following
use homotopy pullbacks and therefore depend on properness of the
underlying $f$-local model structure, this is mentioned most of the
time. 

\begin{lemma}
\label{lem:lem3}
Let $T$ and $f:X\rightarrow Y$ be as in \prettyref{def:univfloc}, and
assume that the $f$-localization of $\simplicial{\topos{T}}$ is
proper. A map $p:E\rightarrow B$ of simplicial sheaves is
universally $f$-local if and only if for all morphisms $g:Z\rightarrow
B$ the following pullback diagram is an $f$-local homotopy pullback: 
\begin{center}
  \begin{minipage}[c]{10cm}
    \xymatrix{
      Z\times_BE\ar[r]\ar[d]_{g^\ast(p)} & E\ar[d]^p\\
      Z\ar[r]_g&B.
    }
  \end{minipage}
\end{center}
\end{lemma}

\begin{proof}
The ``if''-direction is clear, so let $p:E\rightarrow B$ be
universally $f$-local. Factor $p$ as 
$$
p: E\stackrel{j}{\longrightarrow}\widetilde{E}
\stackrel{q}{\longrightarrow} B,
$$
where $j:E\rightarrow \widetilde{E}$ is an $f$-local weak equivalence
and $q:\widetilde{E}\rightarrow B$ is an $f$-local fibration. 
Using properness of the $f$-local model structure, we need to show
that the induced map $Z\times_B 
E\rightarrow Z\times_B\widetilde{E}$ is an $f$-local weak
equivalence. By \prettyref{lem:decompfib}, this map is weakly
equivalent to the map of homotopy colimits
$$
\hocolim_{T\times\mathbf{\Delta}\downarrow Z}\, (g^\ast(p))^{-1}(\sigma)
\rightarrow 
\hocolim_{T\times\mathbf{\Delta}\downarrow Z}\,(g^\ast(q))^{-1}(\sigma),
$$
where the diagram $(g^\ast(p))^{-1}(\sigma)$ is the diagram of the fibres
of $g^\ast(p):Z\times_BE\rightarrow Z$ over the simplices
$\sigma:\Delta^n\times U\rightarrow Z$ of $Z$, and the same for
$(g^\ast(q))^{-1}(\sigma)$. 
But for any simplex $\sigma:\Delta^n\times U\rightarrow Z$ of $Z$, the
induced map 
$$
(g^\ast(p))^{-1}(\sigma)=p^{-1}(g\circ \sigma)\rightarrow
q^{-1}(g\circ\sigma)=(g^\ast(q))^{-1}(\sigma)
$$ 
is an $f$-local weak equivalence, since $p$ was assumed to be
universally $f$-local and the $f$-local model structure was assumed to
be proper. But then the morphism between diagrams
consists of $f$-local weak equivalences only, so the above homotopy
colimit is an $f$-local weak equivalence. This shows the claim. 
\end{proof}

Note that the above result also establishes that the property of being
universally $f$-local is stable under pullbacks.

\begin{corollary}
\label{cor:pbstab}
Under the assumptions of \prettyref{lem:lem3}, if $p:E\rightarrow
B$ is universally $f$-local and $g:Z\rightarrow B$ is any morphism of
simplicial sheaves, then the map $g^\ast(p):Z\times_BE \rightarrow Z$ is
universally $f$-local. 
\end{corollary}

The universally $f$-local maps play the role in the $f$-local
model category of the quasi-fibrations in \cite{dold:thom},
the sharp maps in \cite{rezk:1998:sharp} and the locally trivial
morphisms in \cite{classify} - they are a replacement for honest
fibrations that still give rise to fibre sequences but are easier to
handle. 

\begin{lemma}
\label{lem:lem4}
Let $T$ and $f:X\rightarrow Y$ be as in \prettyref{def:univfloc}, and
assume that the $f$-localization of $\simplicial{\topos{T}}$ is
proper. A map $p:E\rightarrow B$ of simplicial sheaves is
universally $f$-local if and only if for any diagram
\begin{center}
  \begin{minipage}[c]{10cm}
    \xymatrix{
      Z\times_BE \ar[r]^{p^\ast(g)} \ar[d] &
      W\times_BE\ar[r]\ar[d] & E\ar[d]^p\\ 
      Z\ar[r]_g &W\ar[r]&B
    }
  \end{minipage}
\end{center}
with $g$ an $f$-local weak equivalence, the map $p^\ast(g)$ is also an
$f$-local weak equivalence. 
\end{lemma}

\begin{proof}
If $p$ is universally $f$-local, then the outer square and the right
square are $f$-local homotopy pullbacks by \prettyref{lem:lem3}. By
the homotopy pullback lemma \cite[Lemma
II.8.22]{goerss:jardine:1999:simplicial}, the left square is also an
$f$-local homotopy pullback. But then
it is easy to see that $p^\ast(g)$ must be an $f$-local weak equivalence
as well. 

Now assume that the condition is satisfied. By \prettyref{lem:lem3},
it suffices to check that the diagram 
\begin{center}
  \begin{minipage}[c]{10cm}
    \xymatrix{
      Z\times_BE\ar[r]\ar[d]_{g^\ast(p)} & E\ar[d]^p\\
      Z\ar[r]_g&B.
    }
  \end{minipage}
\end{center}
is a homotopy pullback for any map $g:Z\rightarrow B$. 
Factor $g$ as 
$$
g: Z\stackrel{j}{\longrightarrow}W
\stackrel{q}{\longrightarrow} B,
$$
where $j:Z\rightarrow W$ is an $f$-local weak equivalence
and $q:W\rightarrow B$ is an $f$-local fibration. By the assumption,
the morphism $p^\ast(j):Z\times_BE\rightarrow W\times_BE$ is an $f$-local
weak equivalence. Therefore, the above diagram is in fact an $f$-local
homotopy pullback, so $p:E\rightarrow B$ is universally $f$-local.
\end{proof}

The above result states that universally $f$-local maps are exactly
the sharp maps in the sense of Rezk for the $f$-local model structure,
cf. \cite{rezk:1998:sharp}.  It implies in particular that $f$-local
fibrations are universally $f$-local. Moreover, it implies that
\emph{simplicial fibrations over $f$-local bases are universally
  $f$-local:} 

\begin{corollary}
\label{cor:cor6}
Under the assumptions of \prettyref{lem:lem4}, if $p:E\rightarrow
B$ is a simplicial fibration and $B$ is $f$-local and fibrant, then $p$
is universally $f$-local.  
\end{corollary}

\begin{proof}
This is a consequence of \prettyref{lem:lem4} and \cite[Lemma
A.3]{jardine:2000:motivic}. 
\end{proof}

Our goal in the sequel will be to characterize universally $f$-local
maps. 

\begin{theorem}
\label{thm:localization}
Let $T$ be a site and let $f:X\rightarrow Y$ be a
morphism of simplicial sheaves on $\simplicial{\topos{T}}$. Assume
that the $f$-local model structure on $\simplicial{\topos{T}}$ is
proper.   

Let $p:E\rightarrow B$ be a simplicial fibration and let
$j:B\rightarrow L_fB$ be an $f$-local fibrant replacement. Then $p$ is
universally $f$-local if and only if for any $U\in T$ and any simplex
$\sigma:\Delta^n\times U\rightarrow B$, the induced morphism 
$p^{-1}(\sigma)\rightarrow (j\circ p)^{-1}(j\circ\sigma)$ is an
$f$-local weak equivalence. 
\end{theorem}

\begin{proof}
Without loss of generality we can assume that $p$ and $j$ are
simplicial fibrations: first factor $j$ as
$B\rightarrow \widetilde{B}\rightarrow L_fB$ with the first map a
trivial cofibration and the second map a fibration. Then factor the
composition $E\rightarrow B\rightarrow\widetilde{B}$ as $E\rightarrow
\widetilde{E}\rightarrow\widetilde{B}$ with the first map a trivial
cofibration and the second map a fibration. All the factorizations are
done in the simplicial model structure, therefore the replacement
$\widetilde{p}:\widetilde{E}\rightarrow\widetilde{B}$ is universally
$f$-local if and only if $p:E\rightarrow B$ is. 

In the following, we assume that $p$ and $j$ are simplicial
fibrations. We want to show that for each $U\in T$ and each simplex
$\sigma:\Delta^n\times U\rightarrow B$ of $B$, the pullback diagram 
\begin{center}
  \begin{minipage}[c]{10cm}
    \xymatrix{
      p^{-1}(\sigma)\ar[r]\ar[d] & E\ar[d]^p\\
      \Delta^n\times U\ar[r]_\sigma&B.
    }
  \end{minipage}
\end{center}
is an $f$-local homotopy pullback. Consider now the following diagram
\begin{center}
  \begin{minipage}[c]{10cm}
    \xymatrix{
      p^{-1}(\sigma) \ar[rr]^{p^\ast(i)} \ar[d] \ar@{}[drr]|{\txt{\bf{I}}}
      &&(j\circ p)^{-1}(j\circ \sigma)\ar[rr]\ar[d] 
      \ar@{}[drr]|{\txt{\bf{II}}}&&
      E\ar[d]^p\\ 
      \Delta^n\times U \ar[rrd]_{\operatorname{id}} \ar[rr]^i
      &&j^{-1}(j\circ\sigma)\ar[d]_{(j\circ\sigma)^\ast(j)}\ar[rr]
      \ar@{}[drr]|{\txt{\bf{III}}}&&B\ar[d]^j \\ 
      &&\Delta^n\times U\ar[rr]_{j\circ \sigma} &&L_f B.
    }
  \end{minipage}
\end{center}
in which the squares {\bf I}-{\bf III} are pullbacks. The morphism
$i:\Delta^n\times U\rightarrow j^{-1}(j\circ\sigma)$ is a consequence
of the universal property for the pullback square {\bf III}. 

The squares {\bf III} and {\bf II+III} are $f$-local homotopy
pullbacks by \prettyref{cor:cor6}. By the homotopy pullback lemma 
\cite[Lemma II.8.22]{goerss:jardine:1999:simplicial}, the square {\bf
  II} is also an $f$-local homotopy pullback. By the same lemma, the
square {\bf I+II} is an $f$-local homotopy pullback if and only if
{\bf I} is an $f$-local homotopy pullback. 

The map $(j\circ\sigma)^\ast(j)$ is an $f$-local weak equivalence
since {\bf III} is a 
homotopy pullback and $j$ is an $f$-local weak equivalence -- it is (a
simplicial replacement of) the localization morphism $B\rightarrow
L_fB$. By $2$-out-of-$3$, the map $i:\Delta^n\times U\rightarrow
j^{-1}(j\circ\sigma)$ is an $f$-local weak equivalence. But then the
square {\bf I} is an $f$-local homotopy pullback if and only if
$p^\ast(i):p^{-1}(\sigma)\rightarrow (j\circ p)^{-1}(j\circ\sigma)$ is an
$f$-local weak equivalence. The criterion is proved.
\end{proof}

\begin{corollary}
\label{cor:thm8}
Let $T$ be a site and let $f:X\rightarrow Y$ be a
morphism of simplicial sheaves on $\simplicial{\topos{T}}$. Assume
that the $f$-local model structure on $\simplicial{\topos{T}}$ is
proper.   

A morphism $p:E\rightarrow B$ of simplicial sheaves is universally
$f$-local if for any diagram
\begin{center}
  \begin{minipage}[c]{10cm}
    \xymatrix{
      & E \ar[r]^i \ar[d]_p & \widetilde{E} \ar[d]^{\widetilde{p}}\\ 
      Z \ar[r] & B\ar[r]_j & L_fB
    }
  \end{minipage}
\end{center}
with $j:B\rightarrow L_fB$ an $f$-local fibrant replacement,
$\widetilde{p}$ a simplicial fibration and $i$ a simplicial trivial
cofibration, the induced map $Z\times_BE\rightarrow
Z\times_{L_fB}\widetilde{E}$ is an $f$-local weak equivalence.
\end{corollary}

\begin{proof}
We complete the diagram in the statement: 
\begin{center}
  \begin{minipage}[c]{10cm}
    \xymatrix{
      Z\times_{L_fB}\widetilde{E} \ar[r] \ar[d] & E \ar[r]^i \ar[d]_p &
      \widetilde{E} \ar[d]^{\widetilde{p}}\\  
      Z \ar[r] & B\ar[r]_j & L_fB.
    }
  \end{minipage}
\end{center}
Both the  right and the outer square are $f$-local homotopy pullbacks
by \prettyref{cor:cor6}. 
Consider now the following diagram
\begin{center}
  \begin{minipage}[c]{10cm}
    \xymatrix{
      Z\times_{B}E \ar[r] \ar[d] & E \ar[r]^i \ar[d]_p &
      \widetilde{E} \ar[d]^{\widetilde{p}}\\  
      Z \ar[r] & B\ar[r]_j & L_fB.
    }
  \end{minipage}
\end{center}
Then $p$ is universally $f$-local if and only if (in every possible
such situation) the left square of this diagram is an $f$-local
homotopy pullback. But since the right square is an $f$-local homotopy
pullback, this left square is an $f$-local homotopy pullback if and
only if the outer square is. But because of the previous diagram, the
outer square is an $f$-local homotopy pullback if and only if the
induced map $Z\times_BE\rightarrow Z\times_{L_fB}\widetilde{E}$ is an
$f$-local weak equivalence.
\end{proof}

\section{Fibrewise localization}
\label{sec:fibwise}

In this section, we recall several possible definitions of fibrewise 
localization in categories of simplicial sheaves. For a discussion of 
fibrewise localization in the category of topological spaces resp. simplicial 
sets cf. \cite[Section 1.F]{farjoun:1996:cellular} resp. 
\cite[Chapter 6]{hirschhorn:2003:modelcats}. For simplicial sets, one can
define fibrewise localization as follows:

\begin{definition}
Let $L$ be a localization functor on $\simplicial{\set}$. Then $L$ admits a 
fibrewise version if for any fibration  $p:E\rightarrow B$ of 
simplicial sets there exists a commutative diagram
\begin{center}
  \begin{minipage}[c]{10cm}
    \xymatrix{
      E \ar[rr]^{\simeq_f} \ar[rd]_p && \overline{E} \ar[ld]^{\overline{p}}\\
      &B
    }
  \end{minipage}
\end{center}
where $\overline{p}$ is a fibration and $E\rightarrow \overline{E}$ an 
$f$-local weak equivalence, such that for each simplex 
$\sigma:\Delta^n\rightarrow B$ the induced morphism 
$p^{-1}(\sigma)\rightarrow \overline{p}^{-1}(\sigma)$ is (simplicially 
equivalent to) the localization morphism 
$p^{-1}(\sigma)\rightarrow L(p^{-1}(\sigma))$.
\end{definition} 

\begin{remark}
We want to note that pointed and unpointed simplicial sets behave rather
differently with respect to fibrewise localization. For unpointed
simplicial sets, one 
can construct fibrewise localizations in various different ways, whereas for
pointed simplicial sets, one always has to make special connectivity
assumptions on the base resp. the fibre because usually there is no
continuous choice of base point in a nontrivial fibre sequence
$F\rightarrow E\stackrel{p}{\longrightarrow} B$. 
This difference between the unpointed and the pointed setting is
also illustrated by  \cite[Proposition
  6.1.4]{hirschhorn:2003:modelcats}. See also the discussion in \cite[Remark
  1.A.7]{farjoun:1996:cellular}.
\end{remark}

The right translation of this to a simplicial sheaf setting is not exactly 
immediate: the above definition hinges on the fact that $B$ is the homotopy
colimit of its simplices, hence a homotopy colimit of contractible spaces. 
This is no longer true in the simplicial sheaf setting, where $B$ is the
homotopy colimit of simplices $\Delta^n\times U$ but $U$ is typically not
contractible. In the following, we review possible definitions and extensions
of fibrewise localization to simplicial sheaves. 

\subsection{Fibrewise localization after Chataur and Scherer}

Chataur and Scherer have provided a version of fibrewise localization for 
general pointed model categories satisfying some axioms, cf. 
\cite[Theorem 4.3]{chataur:scherer:2003:fibrenull}. 

\begin{theorem}
Let \category{M} be a model category which is pointed, left proper,
cellular and in which the cube axiom and the ladder axiom holds. 
Let $L_f:\category{M}\rightarrow\category{M}$ be a localization
functor which preserves products, and let $p:F\rightarrow E\rightarrow
B$ be a fibre sequence in $\category{M}$.  
Then there exists a fibrewise $f$-localization of $p$. 
\end{theorem}

We note that localization functors of simplicial sheaves commute
with finite products as remarked in the proof of \cite[Lemma
  2.2.32]{morel:voevodsky:1999:a1}, and that cube and ladder axiom 
for categories of simplicial sheaves are consequences of homotopy 
distributivity, cf. \prettyref{sec:classify}. Therefore, the fibrewise
localization method of Chataur and Scherer works in model categories
of simplicial sheaves. 
A result similar to the above can be formulated for fibre sequences
over simply-connected base spaces, replacing the product condition on
$L_f$ by the join axiom, cf. \cite{chataur:scherer:2003:fibrenull}.

Note that the construction of Chataur and Scherer only localizes fibres over
the base point. It is therefore cannot deal with the full generality
of simplicial sheaves. 

\subsection{Fibrewise localization via classifying spaces}
For a locally trivial morphism
$f:E\rightarrow B$ of topological spaces with fibre $F$, one can explain 
quite easily how to construct the fibrewise localization: Take a 
trivialization of $f$, i.e. a covering $U_i$ of $X$ over which 
$f|_{U_i}:E\times_B U_i\cong U_i\times F\rightarrow U_i$. Then apply the 
simplicial coaugmented functor: On the level of the trivialization one 
simply replaces the space $F$ by the space $LF$. On the level of transition
morphisms, one applies the functor $L$ to the transition map. For this
to work we need the functor $L$ to be continuous. This produces an
explicit recipe to construct an $LF$-bundle over $B$. 

The same argument can be applied to locally trivial morphisms of simplicial
sheaves on a site $T$, in the sense of \cite[Definition 3.5]{classify}. One 
can then do the above argument, or use the existence of classifying spaces
for locally trivial morphisms as in \cite{classify}: if the fibre sequence 
$F\rightarrow E\rightarrow B$ is locally
trivial, it is classified by a morphism $B\rightarrow \classify{F}$. 
Composing with the morphism of classifying space induced from the
coaugmentation, we obtain a morphism $B\rightarrow 
\classify{F}\rightarrow \classify{LF}$. Pulling back the universal $LF$-fibre
sequence along this morphism produces an $LF$-fibre sequence $LF\rightarrow
\overline{E}\rightarrow B$ over $B$, which is the fibrewise localization of
the fibre sequence we started with. This implies that in the above
situation any locally trivial $F$-fibre sequence of simplicial sheaves
$F\rightarrow E\rightarrow B$ can be mapped via a homotopy commutative
diagram  
\begin{center}
  \begin{minipage}[c]{10cm}
    \xymatrix{
      F \ar[r] \ar[d] & E \ar[r] \ar[d] & B \ar[d] \\
      LF \ar[r] & \overline{E} \ar[r] & B
    }
  \end{minipage}
\end{center}
to a fibre sequence over $B$, i.e. a fibrewise localization
exists. 

\subsection{Fibrewise localization via homotopy colimit decomposition}
The fibrewise localization for simplicial sets can be defined conveniently
using the homotopy colimit decomposition, which can be viewed as a 
reformulation of the above method for locally trivial morphisms. 
One writes the fibration $p:E\rightarrow B$ as the map of homotopy colimits 
$\hocolim p^{-1}(\sigma)\rightarrow \hocolim\sigma$ over the simplices of the
base and defines the fibrewise localization to be the map of homotopy colimits
$\hocolim L_f(p^{-1}(\sigma))\rightarrow\hocolim\sigma$, cf.
\cite[1.F.3]{farjoun:1996:cellular}.

In the simplicial sheaf setting - because the simplices 
$\sigma:\Delta^n\times U\rightarrow B$ are not contractible - we can not simply
apply the localization functor. We need to discuss in a little more detail 
 what the localization of the fibre should be. We propose the following 
definition which, however, only works if the localized model structure is 
proper. 
Consider the fibre $p^{-1}(\sigma)$ over the simplex 
$\sigma:\Delta^n\times U\rightarrow B$. We apply the localization to this 
morphism and obtain $L_f(p^{-1}(\sigma))\rightarrow L_f(\Delta^n\times U)$.
Now we have localized the fibre, but the base simplex and therefore the whole 
diagram has changed - the homotopy colimit is not necessarily $B$ any more. 
Therefore, we factor $L_f(p^{-1}(\sigma))\rightarrow L_f(\Delta^n\times U)$ as 
a trivial cofibration $L_f(p^{-1}(\sigma))\rightarrow \widetilde{F}$ and a 
fibration $\widetilde{F}\rightarrow L_f(\Delta^n\times U)$. The pullback 
$(\Delta^n\times U)\times_{L_f(\Delta^n\times U)}\widetilde{F}$ is then the best 
approximation to the localization of $p^{-1}(\sigma)$ which still maps to
the (non-local) simplex $\Delta^n\times U$. 
Properness of the local model structure is needed to make sure that the 
morphism 
$$
p^{-1}(\sigma)\rightarrow (\Delta^n\times U)\times_{L_f(\Delta^n\times U)}
\widetilde{F}
$$ 
is an $f$-local weak equivalence. Of course, for a clean definition we need 
to replace statements and arguments involving ``the fibre of 
$p:E\rightarrow B$'' by corresponding statements about ``the diagram of the 
fibres''.

\begin{definition}
\label{def:fibloc1}
Let $T$ be a site, and let $p:E\rightarrow B$ be a morphism of
simplicial sheaves on $T$. We consider the category of
$(\mathbf{\Delta}\times T)\downarrow B$-diagrams in  
$\simplicial{\topos{T}}$ equipped with the model structure which has
the fibrations and weak equivalences from the $f$-local model
structure, and cofibrations defined by left lifting property. 

From the morphism $p:E\rightarrow B$ we obtain a morphism of diagrams
$\mathcal{F}\rightarrow \operatorname{id}$, where 
$$
\mathcal{F}:\sigma\mapsto p^{-1}(\sigma),\qquad
\operatorname{id}:(\sigma:\Delta^n\times U\rightarrow B)\mapsto
\Delta^n\times U.
$$
We then consider the following diagram
\begin{center}
  \begin{minipage}[c]{10cm}
    \xymatrix{
      \mathcal{F} \ar[r] \ar[d]_p & \widetilde{\mathcal{F}}
      \ar[d]^{\widetilde{p}}\\ 
      \operatorname{id} \ar[r]_j & L_f(\operatorname{id}),
    }
  \end{minipage}
\end{center}
where $j$ is a fibrant replacement in the model structure on the
diagram category and $\widetilde{\mathcal{F}}$ is obtained from a
factorization of $j\circ p$ as trivial cofibration
$\mathcal{F}\rightarrow\widetilde{\mathcal{F}}$ followed by a
fibration $\widetilde{p}$. 
The diagram of the ``$f$-localized homotopy fibres'' is then obtained by
the pullback
$\operatorname{id}\times_{L_f(\operatorname{id})}\widetilde{\mathcal{F}}$. This
is a functorial way for assigning to each simplex
$\sigma:\Delta^n\times U\rightarrow B$ the pullback $(\Delta^n\times
U)\times_{L_f(\Delta^n\times U)} L_f(p^{-1}(\sigma))$. 

The fibrewise localization $\overline{p}:\overline{E}\rightarrow B$ is
then defined to be the homotopy colimit of the diagram 
$\operatorname{id}\times_{L_f(\operatorname{id})}\widetilde{\mathcal{F}}$. 
\end{definition}

\begin{remark}
\begin{enumerate}[(i)]
\item
In the special case of simplicial sets, where the base can be
decomposed into (contractible) simplices, this definition reduces to
the usual fibrewise localization. In the simplicial sheaf setting,
where the base can not be decomposed into contractible pieces, we use
the homotopy fibre definition over the simplices, and the homotopy
colimit decomposition to define the fibrewise localization. This
construction somehow sits inbetween the ``classical'' fibrewise
localization and the computation of the homotopy fibre - over
representable objects we have the homotopy fibre, anything more global
than representable objects behaves like fibrewise localization. 
\item Note finally that if $\operatorname{pt}$ denotes the terminal object of
  $\simplicial{\topos{T}}$, then the induced morphism 
  $p^{-1}(\sigma)\rightarrow
  \overline{p}^{-1}(\sigma)$ is the $f$-localization for any simplex
  $\sigma:\Delta^n\times\operatorname{pt}\rightarrow B$. 
\end{enumerate}
\end{remark}

This definition of fibrewise localization has the right properties: it
is an $f$-local weak equivalence on the total spaces $E\rightarrow
\overline{E}$, and on the local fibres it is exactly the ``canonical''
morphism from point-set fibre to ``$f$-localized homotopy fibre rel
base simplex''.

\begin{lemma}
\label{lem:fibwise1}
Let $T$ be a site, let $f:X\rightarrow Y$ be a morphism of simplicial sheaves
such that the $f$-localized model structure is proper. Let $p:E\rightarrow B$
be a fibration of simplicial sheaves. Then the morphism 
$E\rightarrow \overline{E}$ from \prettyref{def:fibloc1} is an $f$-local
weak equivalence and for each simplex $\sigma:\Delta^n\times U\rightarrow B$,
the following is an $f$-local homotopy pullback:
\begin{center}
  \begin{minipage}[c]{10cm}
    \xymatrix{
      \overline{p}^{-1}(\sigma)\ar[r] \ar[d] &
      L_f(\overline{p}^{-1}(\sigma))\ar[d]\\
      \Delta^n\times U \ar[r] & L_f(\Delta^n\times U).
    }
  \end{minipage}
\end{center}
\end{lemma}

\begin{proof}
The fact that the diagrams are $f$-local pullbacks for each simplex of the 
base is a consequence of the definition and properness of the local 
model structure: 
by \prettyref{lem:decompfib}, the fibre $\overline{p}^{-1}(\sigma)$ is the 
space $(\Delta^n\times U)\times_{L_f(\Delta^n\times U)} L_f(p^{-1}(\sigma))$ which is 
defined as the pullback
\begin{center}
  \begin{minipage}[c]{10cm}
    \xymatrix{
      \overline{p}^{-1}(\sigma)\ar[r] \ar[d] &
      \widetilde{F}\ar[d]&  L_f(p^{-1}(\sigma))\ar[l]_{\simeq_f}\\
      \Delta^n\times U \ar[r] & L_f(\Delta^n\times U).
    }
  \end{minipage}
\end{center}
In particular, by properness of the local model structure 
$\overline{p}^{-1}(\sigma) \rightarrow \widetilde{F}$ is an $f$-local weak 
equivalence and $\widetilde{F}$ is $f$-local by definition. The above
pullback is obviously an $f$-local homotopy pullback, and it is  
simplicially equivalent to the one claimed in the diagram. 

Now consider the following commutative diagram
\begin{center}
  \begin{minipage}[c]{10cm}
    \xymatrix{
      p^{-1}(\sigma) \ar[r]\ar[d]_{\simeq_f} & 
      \overline{p}^{-1}(\sigma)\ar[d]^{\simeq_f} \\ 
      L_f(p^{-1}(\sigma)) \ar[r]_{\simeq_f} & \widetilde{F}
    }
  \end{minipage}
\end{center}
arising from the definition of $\overline{p}^{-1}(\sigma)$. The left and bottom
morphism are $f$-local weak equivalences by construction. We saw above
that the right morphism is also an $f$-local weak equivalence by properness. 
Therefore, the top morphism is an $f$-local weak equivalence. The morphism
$E\rightarrow\overline{E}$ is the homotopy colimit of the morphisms
$p^{-1}(\sigma)\rightarrow \overline{p}^{-1}(\sigma)$, and is therefore an
$f$-local weak equivalence.
\end{proof}

The properties established by the lemma above could be seen as an adequate 
definition of fibrewise localization in the simplicial sheaf setting - an 
$f$-local equivalence on the total space and a suitable localization morphism
on the fibres. The lemma also allows us to formulate what it means for a map
to have ``$f$-local fibres''.

\begin{definition}
\label{def:flocfib}
Under the conditions of \prettyref{lem:fibwise1}, a morphism 
$p:E\rightarrow B$ is said to have \emph{$f$-local fibres} if one of the 
following equivalent definitions is satisfied:
\begin{enumerate}[(i)]
\item The morphism $E\rightarrow \overline{E}$ is a simplicial weak 
equivalence.
\item For any simplex $\sigma:\Delta^n\times U\rightarrow B$ the following
induced commutative diagram is a simplicial homotopy pullback:
\begin{center}
  \begin{minipage}[c]{10cm}
    \xymatrix{
      p^{-1}(\sigma)\ar[r] \ar[d] &
      L_f(p^{-1}(\sigma))\ar[d]\\
      \Delta^n\times U \ar[r] & L_f(\Delta^n\times U).
    }
  \end{minipage}
\end{center}
\end{enumerate}
\end{definition}

\begin{lemma}
\label{lem:oldlem45}
Assume the conditions of \prettyref{lem:fibwise1}, let $p:E\rightarrow B$ 
be a morphism in $\simplicial{\topos{T}}$ and let 
$\overline{p}:\overline{E}\rightarrow B$ be its fibrewise localization. 
Then $p$ is universally $f$-local if and only if $\overline{p}$ is.
\end{lemma}

\begin{proof}
It follows from \prettyref{lem:fibwise1} that the following two pullbacks 
are $f$-locally weakly equivalent:
\begin{center}
  \begin{minipage}[c]{10cm}
    \xymatrix{
      p^{-1}(\sigma) \ar[r] \ar[d] & E \ar[d]^p && 
      \overline{p}^{-1}(\sigma) \ar[r] \ar[d] & \overline{E} 
      \ar[d]^{\overline{p}} \\
      \Delta^n\times U \ar[r]_\sigma & B &&
      \Delta^n\times U \ar[r]_\sigma & B
    }
  \end{minipage}
\end{center}
This implies the claim.
\end{proof}

\subsection{Comparison results}

\begin{lemma}
\label{lem:fibcomp}
Let $T$ be a site, let $f:X\rightarrow Y$ be a morphism of simplicial sheaves
on $T$ such that the $f$-local model structure is proper, and let 
$p:E\rightarrow B$ be a fibration of simplicial sheaves. 
\begin{enumerate}[(i)]
\item
Assume there exists a base point $\operatorname{pt}\rightarrow B$. 
There exists a morphism $\overline{E}^{CS}\rightarrow\overline{E}^{HD}$ from
the Chataur-Scherer fibrewise localization to the fibrewise localization 
defined using the homotopy colimit decomposition. This morphism induces 
simplicial weak equivalences over simplices 
$\Delta^n\times\operatorname{pt}\rightarrow B$. 
\item
If $p:E\rightarrow B$ is locally trivial, there exists a morphism 
$\overline{E}^{B}\rightarrow \overline{E}^{HD}$ from the fibrewise localization
using the classifying spaceto the fibrewise localization using
the homotopy colimit decomposition. This morphism is a simplicial weak 
equivalence. 
\end{enumerate}
\end{lemma}

\begin{proof}
We only sketch (i). The Chataur-Scherer fibrewise localization is constructed
as a transfinite telescope in which the successor steps are given by the 
following diagram:
\begin{center}
  \begin{minipage}[c]{10cm}
    \xymatrix{
      F \ar[r]_{\simeq_f} \ar[d] & L_f F \ar[r] \ar[d] & F_1 \ar[d] \\
      E \ar[r]_<<<<<<{\simeq_f} \ar[d]_p & E\cup_F L_fF \ar[r]_>>>>>>\simeq
      \ar[d]^q & E_1  
      \ar[d]^{p_1} \\  
      B \ar[r]_= & B \ar[r]_= & B.
    }
  \end{minipage}
\end{center}
One starts with the fibre sequence $F=p^{-1}(\sigma)\rightarrow E\rightarrow B$, 
takes the localization $F\rightarrow L_fF$ and then the pushout. The resulting 
middle column is not a fibre sequence, so we replace $q$ by a fibration $p_1$,
and let $F_1\rightarrow E_1\rightarrow B$ be the new fibre sequence. It is then
easy to see that the morphism $E\rightarrow \overline{E}^{HD}$ factors through
$E_1\rightarrow \overline{E}^{HD}$. This implies the existence of the required
morphism. For any simplex $\Delta^n\times\operatorname{pt}\rightarrow B$, we 
have an induced composition 
$$
F\rightarrow (\overline{p}^{CS})^{-1}(\sigma)\rightarrow 
(\overline{p}^{HD})^{-1}(\sigma),
$$
where the first morphism and the composition are both $f$-localizations of $F$. 
The second morphism then must be a simplicial weak equivalence as claimed.

(ii) We apply the homotopy colimit decomposition construction of the 
fibrewise localization to the universal locally trivial fibration
$B(\ast,\haut(F),F)\rightarrow \classify{F}$. Using properness, it can
be checked that the result is a locally trivial fibre sequence with fibre 
$L_fF$. The induced morphism $\classify{F}\rightarrow \classify{L_fF}$ is the 
localization morphism, because locally (over simplices where 
$p:E\rightarrow B$ is trivial) the induced morphism is the localization 
morphism. This means that both methods of fibrewise localization agree
on the universal object, so they agree.
\end{proof}

\section{Characterization of universally $f$-local maps}
\label{sec:bfthm}

\begin{lemma}
\label{lem:help}
A map $p:E\rightarrow B$ with $f$-local fibres in the sense of 
\prettyref{def:flocfib} is universally $f$-local if and only if the 
following square is a simplicial homotopy pullback: 
\begin{center}
  \begin{minipage}[c]{10cm}
    \xymatrix{
      E \ar[r] \ar[d]_p & L_fE
      \ar[d]^{L_fp}\\
      B \ar[r] & L_fB
    }
  \end{minipage}
\end{center}
\end{lemma}

\begin{proof}
Assume that the square is a simplicial homotopy pullback. Then up to 
simplicial weak equivalence, the map $p:E\rightarrow B$ is the pullback
of $L_fE\rightarrow L_fB$, which is universally $f$-local by 
\prettyref{cor:cor6}. By \prettyref{cor:pbstab}, it is universally $f$-local.
Note that the map $p:E\rightarrow B$ then has automatically $f$-local fibres, 
by applying the homotopy pullback lemma to the following cube:
\begin{center}
  \begin{minipage}[c]{10cm}
    \xymatrix{
      p^{-1}(\sigma) \ar[rr] \ar[dd] \ar[rd] && E \ar[dd]^<<<<<<p \ar[rd] \\
      & L_f(p^{-1}(\sigma))
      \ar'[r][rr] \ar[dd] && L_fE \ar[dd]^{L_fp} \\ 
      \Delta^n\times U \ar'[r][rr]^\sigma \ar[rd] && B \ar[rd] \\
      & L_f(\Delta^n\times U)
      \ar[rr]_{L_f\sigma} && L_fB 
    }
  \end{minipage}
\end{center}

Now assume that $p:E\rightarrow B$ is universally $f$-local with $f$-local 
fibres. Consider the following diagram:
\begin{center}
  \begin{minipage}[c]{10cm}
    \xymatrix{
      & E \ar[d]^p \ar[r] & L_fE \ar[d]^{L_f(p)} \\
      \Delta^n\times U \ar[r]_\sigma & B\ar[r] & L_fB.
    }
  \end{minipage}
\end{center}
We can assume that $L_f(p):L_fE\rightarrow L_fB$ is a simplicial fibration. 
To check that the square is a simplicial homotopy pullback, it suffices to 
show that for each simplex $\sigma$ as above the induced morphism of fibres
$p^{-1}(\sigma)\rightarrow (L_fp)^{-1}(\sigma)$
is a simplicial weak equivalence. The morphism is an $f$-local weak 
equivalence because $p$ and  $L_f(p)$ are universally $f$-local. 
\begin{center}
  \begin{minipage}[c]{10cm}
    \xymatrix{
      p^{-1}(\sigma) \ar[rr] \ar[dd] \ar[rd] && 
      (L_f(p))^{-1}(\sigma) \ar[dd]^<<<<<<p \ar[rd] \\
      & L_f(p^{-1}(\sigma))
      \ar'[r][rr] \ar[dd] && L_f((L_f(p))^{-1}(\sigma)) \ar[dd]^{L_fp} \\ 
      \Delta^n\times U \ar'[r][rr]_= \ar[rd] && \Delta^n\times U \ar[rd] \\
      & L_f(\Delta^n\times U)
      \ar[rr]_{=} && L_f(\Delta^n\times U)
    }
  \end{minipage}
\end{center}
The front square is a simplicial homotopy pullback because its top morphism
$L_f(p^{-1}(\sigma))\rightarrow L_f((L_f(p))^{-1}(\sigma))$ is the 
localization of an $f$-local weak equivalence, hence a simplicial
weak equivalence. The side squares are both simplicial homotopy pullbacks
because both maps $p$ and $L_f(p)$ have $f$-local fibres. 
The back square is thus a simplicial homotopy pullback, so the  morphism in 
question is a simplicial weak equivalence. 
\end{proof}

The following plays the role of 
\cite[Lemma 3.2]{berrick:farjoun:2003:null}.

\begin{lemma}
\label{lem:old51}
Let $p:E\rightarrow B$ and $j:B\rightarrow C$ be morphisms of
simplicial sheaves with $f$-local $C$. Then $p:E\rightarrow B$
is universally $f$-local if and only if for each simplex
$\sigma:\Delta^n\times U\rightarrow C$ the induced map
$\sigma^\ast(p):(j\circ p)^{-1}(\sigma)\rightarrow j^{-1}(\sigma)$ is
universally $f$-local. 
\end{lemma}

\begin{proof}
By \prettyref{cor:pbstab}, only the ``if''-direction needs proof here.
Consider the following diagram: 
\begin{center}
  \begin{minipage}[c]{10cm}
    \xymatrix{
      (j\circ p)^{-1}(\sigma) \ar[r] \ar[d] & E\ar[d]^p \\
      j^{-1}(\sigma)\ar[d] \ar[r] & B \ar[d]^j\\
      \Delta^n\times U\ar[r]_\sigma & C.
    }
  \end{minipage}
\end{center}
By \prettyref{cor:cor6}, both $j$ and $j\circ p$ are universally
$f$-local. In particular, the outer and the lower square are $f$-local
homotopy pullbacks, so the upper square is an $f$-local homotopy
pullback. Now assume $\tilde{\sigma}:\Delta^n\times U\rightarrow B$
is a simplex of $B$ such that $\sigma=j\circ\tilde{\sigma}$  and
consider the following diagram: 
\begin{center}
  \begin{minipage}[c]{10cm}
    \xymatrix{
      p^{-1}(\tilde{\sigma}) \ar[r] \ar[d]\ar@{}[dr]|{\txt{\bf{I}}}
      &(j\circ p)^{-1}(\sigma) \ar[r] \ar[d] \ar@{}[dr]|{\txt{\bf{II}}}
      & E\ar[d]^p \\
      \Delta^n\times U \ar[r]_i \ar[rd]_{\operatorname{id}}
      & j^{-1}(\sigma)\ar[d] \ar[r] & B \ar[d]^j\\
      &\Delta^n\times U\ar[r]_\sigma & C.      
    }
  \end{minipage}
\end{center}
By assumption, the morphism $\sigma^\ast(p):(j\circ
p)^{-1}(\sigma)\rightarrow j^{-1}(\sigma)$ is universally $f$-local,
so the square {\bf I} is an $f$-local homotopy pullback. By the above,
{\bf II} is an $f$-local homotopy pullback, so the composition {\bf
  I+II} is an $f$-local homotopy pullback. Therefore, $p$ is
universally $f$-local. 
\end{proof}

\begin{theorem}
\label{thm:bfthm}
Let $T$ be a site and let $f:X\rightarrow Y$ be a morphism of
simplicial sheaves in $\simplicial{\topos{T}}$. Assume that the
$f$-local model structure is proper. Let $p:E\rightarrow B$ be a
morphism of simplicial sheaves. 

We denote by $\overline{p}:\overline{E}\rightarrow B$ the fibrewise
$f$-localization of $p$, and by $j:B\rightarrow L_fB$ an $f$-local fibrant 
replacement of $B$.
The following are equivalent, where (iv) only makes sense if
$p:E\rightarrow B$ is locally trivial:
\begin{enumerate}[(i)]
\item The map $p:E\rightarrow B$ is universally $f$-local. 
\item The fibrewise localization $\overline{p}:\overline{E}\rightarrow
  B$ is universally $f$-local.
\item For each simplex $\sigma:\Delta^n\times U\rightarrow L_fB$, the
  following canonical diagram is a simplicial homotopy pullback:
\begin{center}
  \begin{minipage}[c]{10cm}
    \xymatrix{
      (\Delta^n\times U)\times_{L_fB}\overline{E} \ar[r] \ar[d]
      & \overline{p}^{-1}(\sigma) \ar[d] \\
      (\Delta^n\times U)\times_{L_fB}B \ar[r] &
      \Delta^n\times U.
    }
  \end{minipage}
\end{center}
\item For each simplex $\sigma:\Delta^n\times U\rightarrow L_fB$, the
  composition
$$(\Delta^n\times U)\times_{L_fB}B=j^{-1}(\sigma)\rightarrow
B\rightarrow \classify{F} \rightarrow \classify{L_fF}
$$
factors (in the simplicial homotopy category) through the projection
$(\Delta^n\times U)\times_{L_fB}B\rightarrow \Delta^n\times U$.
\end{enumerate}
\end{theorem}

\begin{proof}
We first prove the equivalence between (i) and (ii). Let
$p:E\rightarrow B$ be universally $f$-local. Consider the diagram
\begin{center}
  \begin{minipage}[c]{10cm}
    \xymatrix{
      p^{-1}(\sigma) \ar[rr] \ar[dd] \ar[rd] && E \ar[dd]^<<<<<<p \ar[rd] \\
      & \overline{p}^{-1}(\sigma)
      \ar'[r][rr] \ar[dd] && \overline{E} \ar[dd]^{\overline{p}} \\ 
      \Delta^n\times U \ar'[r][rr]^\sigma \ar[rd]_= && B \ar[rd]^= \\
      & \Delta^n\times U 
      \ar[rr]_\sigma && B 
    }
  \end{minipage}
\end{center}
By definition of fibrewise localization, the morphisms $E\rightarrow
\overline{E}$ and $p^{-1}(\sigma)\rightarrow
\overline{p}^{-1}(\sigma)$ are weak equivalences. Therefore, the front
square is an $f$-local homotopy pullback if and only if the back
square is an $f$-local homotopy pullback. This proves the equivalence
of (i) and (ii).

Assume that $p:E\rightarrow B$ is locally trivial with fibre $F$ in
the sense of \cite[Definition 3.5]{classify}, in particular $p$ is
classified by a morphism $B\rightarrow \classify{F}$. By
\prettyref{lem:fibcomp}, the fibrewise localization can be described
as the pullback of the universal $L_fF$-fibration along the
composition $B\rightarrow \classify{F}\rightarrow\classify{L_fF}$. The
fact that the canonical diagram in (iii) is a simplicial homotopy
pullback is then simply a reformulation  of the fact that the
composition in (iv) factors through the projection
$j^{-1}(\sigma)\rightarrow\Delta^n\times U$. Therefore, (iii) and (iv)
are also equivalent provided (iv) makes sense.

It remains to prove the equivalence of (ii) and (iii). We complete the
diagram in (iii) as follows
\begin{center}
  \begin{minipage}[c]{10cm}
    \xymatrix{
      (\Delta^n\times U)\times_{L_fB}\overline{E} \ar[r] \ar[d]
      & \overline{p}^{-1}(\sigma) \ar[d]
      \ar[r] & L_f(p^{-1}(\sigma))\ar[d] \\
      (\Delta^n\times U)\times_{L_fB}B \ar[r] &
      \Delta^n\times U\ar[r] & L_f(\Delta^n\times U)
    }
  \end{minipage}
\end{center}
The right square is a simplicial homotopy pullback by the definition
of fibrewise localization, cf. \prettyref{lem:fibwise1}. The right
vertical map is simplicially equivalent to the localization of the
left vertical map. Therefore, by 
\prettyref{lem:help}, the outer square is a simplicial homotopy
pullback if and only if the left vertical map is universally
$f$-local. The left square is a simplicial homotopy pullback if and
only if (ii) holds. For the equivalence of (ii) and (iii) it then
suffices to show that $\overline{p}:\overline{E}\rightarrow B$ is
universally $f$-local if and only if for each simplex
$\sigma:\Delta^n\times U\rightarrow L_fB$, the induced map $(j\circ
\overline{p})^{-1}(\sigma)\rightarrow j^{-1}(\sigma)$ is universally
$f$-local. This is the statement of \prettyref{lem:old51}.
\end{proof}

\begin{remark}
The above result can probably not be effectively used for showing that a
given map is universally $f$-local. However, it explains
philosophically why a map can fail to be universally $f$-local. In
spite of added complication of considering all the local fibres of
$p$, the reason is still the same as in
\cite{berrick:farjoun:2003:null}: a map fails to be universally
$f$-local if its fibrewise localization is non-trivial over non-local
parts. In the simplicial situation, one pulls back the fibrewise
localization to the fibre $A_{L_fB}$ of $j:B\rightarrow L_fB$. In the
sheaf situation, one has to replace the single space $A_{L_fB}$ by the
set of all the fibres of $B\rightarrow L_fB$ over the various
simplices. Note that the above result reduces exactly to 
\cite[Theorem 4.1]{berrick:farjoun:2003:null} for
$T=\operatorname{pt}$. 
\end{remark}

\section{Application: fibrations in $\Ao$-homotopy  theory}  
\label{sec:appl}

In this section, we apply the localization theory developed earlier to
discuss  fibrations in $\Ao$-homotopy theory. Hence we specialize to
the site $T=\smk$ of smooth finite type schemes over a field $k$
equipped with the Zariski or Nisnevich topology. We consider the
injective model structure on the category of simplicial sheaves
$\simplicial{\topos{\smk}}$, and apply a Bousfield 
localization to the scheme $\mathbb{A}^1$ considered as constant
representable simplicial sheaf. More details on the construction of
$\Ao$-homotopy theory can be found in \cite{morel:voevodsky:1999:a1}. 

Now recall from \cite{classify}, that for each simplicial sheaf $F$,
there is a classifying space of locally trivial maps with fibre $F$ in
the sense of \cite[Definition 3.5]{classify}. We denote this space by
$B\haut(F)$, since \cite[Theorem 5.10]{classify} shows that this space
can be constructed as the classifying space of the simplicial sheaf of
monoids $\haut(F)$ of homotopy self-equivalences of $F$. We assume
here that the morphisms considered are locally trivial in the
Nisnevich topology. Note that in the above, we are working in the
unpointed category, so we can not talk about fibre sequences in the
sense of \cite[Definition 6.2.6]{hovey:1998:modelcats}. Also the
classification result cited is a classification in the unpointed
setting. 

The main general result is the following.

\begin{theorem}
\label{thm:blocal}
Let $X$ be a cofibrant and $\Ao$-local fibrant simplicial sheaf on
$\smk$. Then  $B\haut(X)$ is $\Ao$-local if and only if the sheaf of 
homotopy self-equivalence groups $\pi_0(\haut(X))$ is strongly
$\Ao$-invariant. 
\end{theorem}

\begin{proof}
(i) We first prove that the simplicial sheaf of monoids of
homotopy self-equivalences $\haut(X)$ is fibrant and $\Ao$-local. 

By \cite[Lemma I.1.8]{morel:voevodsky:1999:a1}, there is a fibration 
\begin{displaymath}
\inthom(X,Y)\rightarrow Y
\end{displaymath}
if $Y$ is fibrant. Thus $\inthom(X,X)$ is fibrant if $X$ is fibrant. 
The simplicial set $\haut(X)(U)$ is a union of connected components of
$\inthom(X,X)(U)$. By 2-out-of-3 for weak equivalences a morphism
$f:X\times U\rightarrow X\times U$ is a weak equivalence if it is
homotopic to a morphism $f':X\times U\rightarrow X\times U$ which is a
weak equivalence. Therefore $\haut(X)(U)$ consists exactly of the
union of the components of $\inthom(X,X)(U)$ which contain weak
equivalences. 

Consider now the  commutative diagram
\begin{center}
  \begin{minipage}[c]{10cm}
    \xymatrix{
      \haut(X)(U)\ar[r] \ar[d] &
      \haut(X)(U\times\Ao) \ar[d] \\
      \inthom(X,X)(U) \ar[r] &
      \inthom(X,X)(U\times\Ao).
    }
  \end{minipage}
\end{center}
The vertical arrows are the inclusions as described above, and the
lower horizontal morphism is a weak equivalence of simplicial sets
since we noted that $\inthom(X,X)$ is $\Ao$-local. In particular, the
lower morphism induces a bijection on the connected components. This
bijection restricts to a bijection between the components consisting
of weak equivalences: first, any morphism $f:X\times U\rightarrow
X\times U$ is a retract of $f\times\id:X\times U\times\Ao \rightarrow
X\times U\times \Ao$, therefore the preimage of a component in
$\haut(X)(U\times\Ao)$ is in $\haut(X)(U)$. Similarly, if $f$ is a
weak equivalence, then $f\times\id$ is a weak equivalence. 
But then the morphism $\haut(X)(U)\rightarrow \haut(X)(U\times\Ao)$ is
a weak equivalence because it is a bijection on connected components,
and the connected components are connected components of the mapping
spaces $\inthom(X,X)$, where we have a weak equivalence. This implies
that $\haut(X)$ is $\Ao$-local if $X$ is $\Ao$-local.

(ii)
By \cite[Lemma 5.44, Theorem 5.45]{morel:2006:a1algtop}, $B\haut(X)$ is
$\Ao$-local if and only if the sheaf of groups $\pi_0 L_{\Ao}\Omega
B\haut X$ is strongly $\Ao$-invariant. The theorem follows, if we can
prove that the obvious morphism 
\begin{displaymath}
\haut X\rightarrow \Omega B\haut X\rightarrow L_{\Ao}\Omega B\haut X 
\end{displaymath}
induces an isomorphism of sheaves of groups $\pi_0$.
But the obvious morphism 
\begin{displaymath}
\haut X\rightarrow \Omega B\haut X
\end{displaymath}
is already a weak equivalence of simplicial sheaves, because the
stalks of $\haut X$ are monoids of homotopy self-equivalences of 
simplicial sets which are group-like. Therefore, the morphism induces 
weak equivalences on the stalks, cf.  \cite[Corollary
IV.1.68]{rudyak:1998:cobordism}. 
This implies that $\Omega B\haut X$ is already $\Ao$-local, hence the
localization $\Omega B\haut X\rightarrow L_{\Ao}\Omega B\haut X$ is a
simplicial weak equivalence. 
\end{proof}

This result has the following consequence. Note that in the following,
we are talking about locally trivial maps, so the fibrewise
localization can be defined on classifying spaces. Note also that the
statement ``the map $p:E\rightarrow B$ is universally $\Ao$-local''
implies that for any choice of base point
$x:\operatorname{Spec}k\rightarrow B$, the resulting sequence
$p^{-1}(x)\rightarrow E\rightarrow B$ is a fibre sequence with
$p^{-1}(x)\simeq_{\Ao} F$. 

\begin{corollary}
\label{cor:loc}
Let $X$ be a cofibrant, fibrant and
$\Ao$-local simplicial sheaf on $\smk$ such that $\pi_0(\haut(X))$ is
strongly $\Ao$-invariant. 

Then we have the following statements:
\begin{enumerate}[(i)]
\item The morphism  
\begin{displaymath}
B(\ast,\haut X,X)\rightarrow B\haut X 
\end{displaymath} 
is universally $\Ao$-local. 
\item Any Nisnevich locally trivial morphism $E\rightarrow B$ whose
  fibre $F$ has the $\Ao$-homotopy type of $X$ is also universally
  $\Ao$-local.  
\item
Denoting by $\mathcal{H}^{\Ao}(Y,X)$ the pointed set of Nisnevich
locally trivial fibre sequences over $Y$ with fibre $X$ up to
$\Ao$-equivalence, we have a natural bijection
\begin{displaymath}
\mathcal{H}^{\Ao}(-,X)\cong [-,B\haut X]_{\Ao}
\end{displaymath}

\end{enumerate}
\end{corollary}

\begin{proof}
(i) \prettyref{cor:cor6} implies
  that the universal fibre sequence is $\Ao$-local if (a simplicial
  fibrant replacement of) the classifying space $B\haut X$ is
$\Ao$-local. But $B\haut X$ is $\Ao$-local since the conditions of
\prettyref{thm:blocal} are satisfied.

(ii) follows from \prettyref{cor:pbstab}. Any Nisnevich locally
trivial fibre sequence is a pullback of the universal fibre sequence
with fibre $F$ along some morphism $B\rightarrow B\haut F$. But from
(i) it follows that the universal fibre sequence over $B\haut
L_{\Ao}F\simeq B\haut X$ is $\Ao$-local. 

For (iii) we first note that \cite[Theorem 5.10]{classify} yields a
bijection  
\begin{displaymath}
\mathcal{H}(-,X)\cong [-,B\haut X].
\end{displaymath}
For the definition of $\mathcal{H}$, cf. \cite[Definition
5.1]{classify}. Since $B\haut X$ is $\Ao$-local, we also have a bijection 
\begin{displaymath}
[-,B\haut X]\cong [-,B\haut X]_{\Ao}.
\end{displaymath}
On the other hand, since $B\haut X$ is $\Ao$-local, the classifying
morphism $B\rightarrow B\haut X$ factors up to homotopy through a
morphism $L_{\Ao}B\rightarrow B\haut X$. By
\prettyref{cor:cor6}, we can hence assume that the fibre
sequence classified by this consists of $\Ao$-local
spaces. Since a morphism between local spaces is an $\Ao$-weak
equivalence if and only if it is a simplicial weak equivalence, the
two equivalence notions for fibre sequences coincide, and we have the
final bijection $\mathcal{H}^{\Ao}(-,X)\cong \mathcal{H}(-,X)$.
\end{proof}

\begin{remark}
Weaker versions of the above have been used in
  \cite{morel:2006:a1algtop} and \cite{torsors} to produce fibre
  sequences from torsors under algebraic groups. 
The above statement can be used to produce classifying spaces for many
other ``fibre sequences'' in $\Ao$-homotopy theory where the structure
groups are no longer algebraic groups. One particularly 
interesting such classifying spaces would be the classifying space of
spherical fibrations: let $S^{2n,n}=S^n\wedge\mathbb{G}_m^{\wedge n}$
be an $\Ao$-local model of the $(2n,n)$-sphere. Then the Nisnevich
locally trivial morphisms of simplicial sheaves with fibre $S^{2n,n}$
are classified by $B\haut S^{2n,n}$. This remains true in the
$\Ao$-local situation if the sheaf of homotopy self-equivalences
$\pi_0\haut S^{2n,n}$ of the $(2n,n)$-sphere is strongly
$\Ao$-invariant. By the computations in \cite[Corollary 5.42, Theorem 
6.36]{morel:2006:a1algtop}, the homotopy endomorphisms of $S^{2n,n}$
are given by the Grothendieck-Witt ring $GW(k)$ for $n\geq 2$ and an
extension of $GW(k)$ for $n=1$. The homotopy self-equivalences are
then the units in the above rings. 

If the sheaf of units of the Grothendieck-Witt sheaf $GW$ are strongly
$\Ao$-invariant, then $B\haut S^{2n,n}$ is an $\Ao$-local classifying
space for spherical fibrations. Unconditionally, its universal
$\Ao$-covering -- the classifying space of the connected component of
$\haut S^{2n,n}$ -- is an $\Ao$-local classifying space for orientable
spherical fibrations. 

There are several interesting directions to pursue here:
\begin{enumerate}[(i)]
\item How does the notion of orientability coming from spherical
  fibrations relate to other notions of orientability in
  $\Ao$-homotopy theory? 
\item I would expect that the classifying space for orientable
  spherical fibrations is cellular with a cell structure similar to
  the one known in ``classical algebraic topology''. This would imply
  that the characteristic classes of orientable spherical fibrations
  over an algebraically closed field coincide with the known
  topological characteristic classes. 
\item There is an obvious morphism $BGL_n\rightarrow B\haut S^{2n,n}$
  obtained from a change-of-fibre along $\mathbb{A}^n\rightarrow
  \mathbb{A}^n/(\mathbb{A}^n\setminus\{0\})\simeq S^{2n,n}$ -- the
  classifying space version of the J-homomorphism. This could possibly
  be used in connection with the characteristic classes in (ii) to
  exhibit simplicial sheaves with a reasonably behaved spherical
  fibration which are not $\Ao$-weakly equivalent to any smooth
  projective scheme. 
\end{enumerate}
\end{remark}

Finally, I would like to remark that an $f$-local version of homotopy 
distributivity does not hold: a homotopy colimit of universally
$f$-local maps is not necessarily universally $f$-local. In
particular, it is not necessarily true that a map which is locally
trivial is universally $f$-local. As an example, let $G$ be a sheaf of
groups on $\smk$ which is $\Ao$-invariant but not strongly
$\Ao$-invariant. Then the map $EG\rightarrow BG$ is not universally
$f$-local - its simplicial homotopy fibre is $G$ (which is $\Ao$-local
by assumption), but its $\Ao$-homotopy fibre is
$\pi_1(L_{\Ao}BG)$. If $\pi_1BG\cong\pi_1(L_{\Ao}BG)$, then $G$ would
be strongly $\Ao$-invariant, contradicting the assumption. In
particular, it seems that the condition on strong $\Ao$-invariance of
self-equivalences in \prettyref{cor:loc}  can not be dropped.

\end{document}